\documentclass[preprint, 12pt]{elsarticle}

\usepackage{amsthm,amsmath, amscd, amssymb, graphicx, enumerate, bbm, latexsym, textcomp, lipsum, extarrows, units, xfrac, nicefrac, fancyhdr, fancyheadings}

\usepackage{lineno}

\numberwithin{equation}{section}
\theoremstyle{plain}

\newtheorem{lemma}{Lemma}[section]
\newtheorem{proposition}{Proposition}[section]
\newtheorem{definition}{Definition}[section]
\theoremstyle{remark}
\newtheorem{remark}{Remark}[section]
\newcommand{\assign}{:=}

\newcommand{\nocomma}{}

\newcommand{\tmop}[1]{\ensuremath{\operatorname{#1}}}

\journal{arXiv}

\begin{document}

\begin{frontmatter}



\title{Non-Analytic Solution to the Fokker-Planck Equation of Fractional Brownian Motion via Laplace Transforms\tnoteref{label1}}
\tnotetext[label1]{It is instructive to mention that portions of this paper was submitted to Mahidol University as part of a research project while the author was a student there.}


\author{Visant Ahuja}

\address{Mahidol University International College, Mahidol University\\ Present address: D.S. Tower 1 Room 18A2, 98 Dangudom Soi, Sukhumvit 33 Road\\ Klongton, Wattana\\ Bangkok 10110, Thailand}
 \ead{visant@alumni.unc.edu}

\begin{abstract}
This paper derives the \emph{non-analytic} solution to the Fokker-Planck equation of fractional Brownian motion using the method of Laplace transform. Sequentially, by considering the fundamental solution of the non-analytic solution, this paper obtains the transition probability density function of the random variable that is described by the It{\^ o}'s stochastic ordinary differential equation of fractional Brownian motion. Furthermore, this paper applies the derived transition probability density function to the Cox-Ingersoll-Ross model governed by the fractional Brownian motion instead of the usual Brownian motion.

\end{abstract}


\end{frontmatter}
\pagestyle{fancy}
\lhead{Probability Function of Fractional Brownian Motion} \chead{} \rhead{}
\renewcommand{\headrulewidth}{0.4pt}


\section{Introduction}\label{intro}

It was in the year 1951 when Feller first wrote his paper, \textit{Two Singular Diffusion Problems} \citep{Feller}, which aims to study the parabolic partial differential equation
\begin{eqnarray}
  u_t = (a x u)_{x x} - ((b x + c) u)_x \label{FellerPDE}, \,\, 0 < x < \infty, 
\end{eqnarray}
where $u = u (t, x)$ and $a, b, c$ are constants with $a > 0$. In his paper, he was able to successfully derive its \emph{analytic} solution along with proving its existence \emph{and} uniqueness. As it turns out, this parabolic partial differential equation is a specific form of the \textit{Fokker-Planck equation} of the diffusion process (or famously known as the Brownian motion). In this paper, we will refer to this specific form simply as the Fokker-Planck equation of Brownian motion. In statistical mechanics, the \textit{Fokker-Planck equation} is a partial differential equation [PDE] that describes the time evolution of the probability density function of the dynamics of a particle that is influenced by random processes (e.g., Brownian motion). Fokker-Planck equation is also known as the Kolmogorov forward equation. Due to the extensive applications of the Brownian motion, it is no doubt that Feller's paper, \textit{Two Singular Diffusion Problems}, have gained a lot of fame. One application that the author is interested in is its use in developing the fuzzy stochastic volatility model in the field of mathematical finance (e.g., see \citep{1}), where one of the equations in the original stochastic volatility model is described by the Fokker-Planck equation of Brownian motion. For the sake of generality, since Brownian motion is a special case of fractional Brownian motion [fBm], the author tries to extend the stochastic volatility model to the case where the Fokker-Planck equation of Brownian motion is {\emph{replaced}} by the Fokker-Planck equation of fBm. On one hand, the Fokker-Planck equation of fBm has already been derived by {\" U}nal \citep{Unal} (and this will be introduced in later sections). On the other hand, in order to develop a ``fuzzy" version of the extended stochastic volatility model, it is necessary to use the solution to the Fokker-Planck equation of fBm. Thus, it is compulsory that we find the solution to the Fokker-Planck equation of fBm before being able to develop the ``fuzzy'' version of the model, and this is the underlying motivation that the author writes this paper.

This paper can be outlined as follows. Section \ref{FPfBm} sets up the general form of the Fokker-Planck equation of fBm such that it has a similar set up as \citep{Feller}. Section \ref{somedef} gives the definitions and some preliminaries necessary before deriving the solution to the Fokker-Planck equation of fBm. Section \ref{derivation} will derive the general form of the non-analytic solution to the Fokker-Planck equation of fBm via Laplace transforms. Then, section \ref{proofs} will impose conditions of the general form of the non-analytic solution to the Fokker-Planck equation of fBm and prove its existence. Section \ref{solution} will provide the fundamental non-analytic solution to the Fokker-Planck equation of fBm based on proven claims in the preceding section. Section \ref{application} will apply the derived solution to the Cox-Ingersoll-Ross [CIR] model, where the Brownian motion in the model is replaced by fBm. Finally, section \ref{conclusion} will conclude this paper and provide some topics that may be further investigated.

\section{The Fokker-Planck Equation of Fractional Brownian Motion}\label{FPfBm}

We devote this section to the set up of the general form of the Fokker-Planck equation of fBm since it will be of major use in the next sections.

First, we note that {\" U}nal \citep{Unal} derives the Fokker-Planck Equation of fBm for \emph{vector-valued} variable by considering It{\^ o}'s stochastic ordinary differential equation of the form
\begin{eqnarray*}
d x_i = f_i (\mathbf{x}, t) d t + g_{i \alpha} (\mathbf{x}, t) d B^H_{\alpha}, \,\, 1 \leqslant i \leqslant n \text{; } 1 \leqslant \alpha \leqslant r.
\end{eqnarray*}
Then, from It{\^ o}'s formula for fBm for a {\emph{scalar}} function $h (\mathbf{x})$,
\begin{equation}
d h (\mathbf{x}) = \left( f_j  \frac{\partial h}{\partial x_j} + H t^{2 H - 1} g_{j \alpha} g_{k \alpha}  \frac{\partial^2 h}{\partial x_j \partial x_k} \right) d t + g_{j \alpha}  \frac{\partial h}{\partial x_j} d B_{\alpha}^H \label{scalarh},
\end{equation}
{\" U}nal derives, through the extensive use of properties of expectations, the Fokker-Planck equation for fBm as
\begin{equation}
\frac{\partial p}{\partial t} + \frac{\partial f_j p}{\partial x_j} - H t^{2 H - 1}  \frac{\partial^2 g_{j \alpha} g_{k \alpha} p}{\partial x_j \partial x_k} = 0, \label{genFPK}
\end{equation}
where
\begin{itemize}
\item $\mathbf{x} \in \mathbb{R}^n$ is a vector,
\item $p\assign p(\mathbf{x},t)$ is the probability density function,
\item $f_j \assign f_j (\mathbf{x}, t)$ is a drift vector,
\item $g_{i \alpha} (\mathbf{x}, t)$ is a diffusion matrix for all $i=j,k$,
\item $d B^H_{\alpha}$ is increment of fBm, and  
\item $H\in(0,1)$ is the Hurst parameter.
\end{itemize}

\subsection{The Fokker-Planck Equation of Fractional Brownian Motion for Scalar-Valued Variable}

The scope of this paper is to derive the non-analytic solution to the Fokker-Planck equation of fBm for \emph{scalar-valued} variable. Therefore, in this \emph{short} subsection, we present the Fokker-Planck equation of fBm for \emph{scalar-valued} variable.

In order to transform the variable in equation (\ref{genFPK}) to scalar-valued variable, to be consistent with the mathematical finance context as well as following Hsu's scheme of substitution \citep{Hsu}, and to be consistent with equation (\ref{FellerPDE}), we let
\begin{itemize}
\item $p (\mathbf{x}, t) = P (X, t) = u(t,x)=u$,  
\item $f_j (\mathbf{x}, t) ={\mu} (X, t) = (bx+c)$,  
\item $g_{j \alpha} = g_{k \alpha} = \sigma (X_t, t), \sigma^2(X_t,t)=ax, a \in \mathbb{R}_{>0}$, and
\item $H=v\in(0,1)$.
\end{itemize}
Then, since the parameter (constant) $H=v$ is \emph{not} restricted to any particular real number in the interval $(0, 1)$, equation (\ref{genFPK}) can be rewritten as
\begin{eqnarray}
u_t = (a t^{2 v - 1} x u)_{x x} - ((b x + c) u)_x \label{mainPDE}, \,\, 0 < x < \infty, 
\end{eqnarray}
where $a, b, c, v$ are constants with $0 < v < 1$. Moreover, note that
\begin{itemize}
\item $u = u (t, x)$ depends on both $x$ \emph{and} $t$,  
\item $x$ depends on $t$,  
\item $a$ depends on $v$ (but we will refer to $a$ as a constant since $v$ depends on neither $x$ \emph{nor} $t$), and  
\item $t$ depends on the parameter (constant) $v$.
\end{itemize}
Now, in order to be even more consistent with equation (\ref{FellerPDE}), we {\emph{relax}} the boundedness condition on $v$ such that $v > 0$.

\begin{remark}
When $v = \nicefrac{1}{2}$ (i.e., the Hurst exponent $H = \nicefrac{1}{2}$), equation (\ref{mainPDE}) reads as
\begin{eqnarray*}
u_t & = & (a t^{2 v - 1} x u)_{x x} - ((b x + c) u)_x = (a x u)_{x x} - ((b x + c) u)_x,
\end{eqnarray*}
where $0 < x < \infty$, and this is equivalent to equation (\ref{FellerPDE}).
\end{remark}

From this point on, this paper will focus on equation (\ref{mainPDE}) [with the condition that $v>0$ and \emph{not} $0<v<1$] and derives the non-analytical solution to equation (\ref{mainPDE}) along with proving the existence of the derived non-analytical solution.

\section{Definitions and Preliminaries prior to Solving the Fokker-Planck Equation of Fractional Brownian Motion}\label{somedef}
We devote this section to providing the readers with some definitions and preliminaries before actually solving the Fokker-Planck equation of fBm. We will follow closely with Feller's \textit{Two Singular Diffusion Problem} \citep{Feller}. Thus, the readers are encouraged to refer to Feller's work while reading through this section. For convenience, the equation that we will dedicate this and the following sections to is presented here again. That is, the equation is
\begin{eqnarray}\label{mainFPK}
u_t = (a t^{2 v - 1} x u)_{x x} - ((b x + c) u)_x, \,\, 0 < x < \infty, 
\end{eqnarray}
where $u = u (t, x)$ and $a, b, c, v$ are constants with $a>0$ and $v > 0$. Notice that this equation is a \textit{linear} parabolic equation. This can be seen more clearly as follows:
\begin{eqnarray*}
u_t & = & (a t^{2 v - 1} x u)_{x x} - ((b x + c) u)_x\\ 
& = & a t^{2 v - 1} x u_{x x} + (2 a t^{2 v - 1} - b x - c) u_x - b u\\
\Longrightarrow b u & = & a t^{2 v - 1} x \frac{\partial^2 u}{\partial x^2} + (2 a t^{2 v - 1} - b x - c) \frac{\partial u}{\partial x} - \frac{\partial u}{\partial t}\\ 
& = & \mathcal{F} (t, x) \frac{\partial^2 u}{\partial x^2} +\mathcal{G} (t, x) \frac{\partial u}{\partial x} - \frac{\partial u}{\partial t},
\end{eqnarray*}
where $\displaystyle{\mathcal{F} (t, x) \assign a t^{2 v - 1} x}$, $\displaystyle{\mathcal{G} (t, x) \assign 2 a t^{2 v - 1} - b x - c}$, and $b u$ is a \emph{linear} function in $u$ (where $u$ is a \emph{linear} function in $x$). It is worthy to remark that the Laplace transform will be done with respect to $x$.

First, it is customary to define what it means for an equation to be a
solution of equation (\ref{mainFPK}).

\begin{definition}[Solution to Equation
(\ref{mainFPK})]
$u (t, x)$ is a solution of equation (\ref{mainFPK}) if, for $x > 0$, it has continuous partial derivatives satisfying equation (\ref{mainFPK}). Also, if for every fixed $s > 0$ and $t > 0$, the functions $e^{- s x} u (t, x)$ and $e^{- s x} u_t (t, x)$ are integrable over $0 < x < \infty$, and this is uniform in every interval $0 < t_o \leqslant t \leqslant t_1 < \infty$.
\end{definition}

Furthermore, we let
\begin{equation}
\omega (t, s) = \int_0^{\infty}{ e^{- s x} u (t, x)\,d x} \label{omega(ts)}.
\end{equation}
The Laplace transform of $u_t (t, x)$ exists and can be found by differentiating equation (\ref{omega(ts)}) formally. Also, note that it is \emph{not} necessary that the Laplace transform of $u_x (t, x)$ exists since, in general, $u (t, x) \rightarrow \infty$ as $x \rightarrow 0$.

Next, for $u_t \in L (0, 1)$, it can be seen that
\begin{eqnarray}
& & \lim_{x \rightarrow 0} \int_x^1{ u_t d x}  = \lim_{x \rightarrow 0}
\int_x^1{ (a t^{2 v - 1} x u)_{x x} - ((b x + c) u)_x\, d x} \nonumber \\
& \Longrightarrow & \lim_{x \rightarrow 0} \{ (a t^{2 v - 1} x u)_x - (b x + c) u \} = - f (t) \label{limx0}
\end{eqnarray}
exists and is bounded in every finite $t$-interval since $u(t,x) \rightarrow \infty$ as $x \rightarrow 0$. 
In this paper,
\textit{$f (t)$ will be called the flux of $u$ at the origin} with same reasoning as Feller \citep{Feller}(i.e., from equation (\ref{mainFPK}),
$\displaystyle{\frac{\partial}{\partial t}  \int_{\alpha}^{\beta} u (t, x)\, d x}$ equals the flux at $\alpha$ minus the flux at $\beta$).

To avoid special considerations for the fundamental solutions, we allow \textit{initial values} to also be discontinuous functions \citep{Feller}.

Now, let $P (x)$ be a function of bounded variation on the interval $(0, \infty)$ such that $P (x) \rightarrow 0$ as $x \rightarrow 0$, and assume its Laplace transform
\begin{equation}
\pi (s) = \int_0^{\infty} e^{- s x}\, d P (x) \label{pis}
\end{equation}
exists.

Additionally, we shall have the following definition:

\begin{definition}[($u (t, x)$ Having Initial Values $P (x)$)]
The solution $u (t, x)$ has the initial values $P (x)$ if
\[ \lim_{t \rightarrow 0} \int_0^x u (t, x)\, d x = P (x) \]
at every point of continuity of $P (x)$.
\end{definition}

\begin{remark} 
This definition makes it necessary that $\omega (t, s) \rightarrow \pi (s)$.
\end{remark}

Likewise, we should define what a \textit{fundamental solution} is.

\begin{definition}[Fundamental Solution]
A solution $u (t, x ; \xi)$ depending on the parameter $\xi > 0$ is a fundamental solution if it assumes the initial values
\[ \left\{ \begin{array}{lll}
     0 & \text{when} & x < \xi\\
     1 & \text{when} & x > \xi
   \end{array} \right.. \]
That is, if
\[ \omega (t, x ; \xi) \rightarrow e^{- s \xi} . \]
\end{definition}

In this instance, provided that the regularity conditions imposed on
our solutions are satisfied by $u (t, x ; \xi)$ uniformly with respect to $\xi$,
\[ u (t, x) = \int_0^{\infty} u (t, x ; \xi)\, d P (\xi) \]
assumes the initial values $P (x)$.

Prior to ending this section, readers should note that the author have used the notation,
\[\widetilde{\kern-7pt\int}{f},\]
in this paper to denote the indefinite integral of any function $f$ with the constant of integration being omitted. More specifically, if 
\[\int{f(x)\,dx} = F(x) + C,\]
where $C$ is the constant of integration, then
\[\widetilde{\kern-7pt\int}{f(x)\,dx} = F(x).\]

\section{Laplace Transforms of the Fokker-Planck Equation of Fractional Brownian Motion}\label{derivation}

In this section, we will take the Laplace transform of equation (\ref{mainFPK}).

Since the individual terms on the right hand side is {\emph{not}} necessarily integrable, we have to be cautious when taking the Laplace transforms of equation (\ref{mainFPK}). However, by assumption, the Laplace transform of the left hand side converges for $s > 0$ \citep{Feller}. Then, by treating the right hand side as a {\emph{unit}} and by using equation (\ref{limx0}), this yields
\begin{equation}
\omega_t(t,s) \assign \omega_t = f (t) + s \int_0^{\infty} e^{- s x}  \{ a t^{2 v - 1}  (x u)_x - (b x + c) u \}\, d x \label{Lpmain}.
\end{equation}
Let $\mathfrak{L}$ be the Laplace transform operator. Then, due to the reason that $\mathfrak{L}\{u\}$ and $\mathfrak{L}\{x u\}$ converge for $s > 0$, $\mathfrak{L}\{(x u)_x\}$ also converges for $s > 0$. This implies that $(x u)_x$ is absolutely integrable near $x = 0$ and
\[ \lim_{x \rightarrow 0} x u (t, x) = \kappa \]
for some constant $\kappa$. It is {\emph{necessary}} that $\kappa = 0$, otherwise $u (t, x)$ would not be integrable.

From equation (\ref{Lpmain}), we have that
\begin{eqnarray}
\omega_t & = & f (t) + a t^{2 v - 1} s \int_0^{\infty}{ e^{- s x}  (x u)_x\, d x} \label{mLpmain} \\
& & - b s \int_0^{\infty}{ e^{- s x} x u\, d x} - c s \int_0^{\infty}{ e^{- s x} u\, d x}.\nonumber
\end{eqnarray}
From equation (\ref{mLpmain}), since $\displaystyle{\frac{\partial}{\partial s} \left\{e^{- s x} x u\right\}}$ exists and is continuous, if we let
\begin{eqnarray*}
\omega & \assign & \int_0^{\infty} e^{- s x} u\, d x\\
&  & \\
\Longrightarrow \omega_s & = & - \int_0^{\infty} e^{- s x} x u\, d x \,\,\,\text{(by applying Leibniz integral rule)}\\
&  & \\
\Longrightarrow s \omega_s & = & - s \int_0^{\infty} e^{- s x} x u\, d x = - \int_0^{\infty} e^{- s x}  (x u)_x\, d x
\end{eqnarray*}
then equation (\ref{mLpmain}) can be rewritten as
\begin{eqnarray}
\omega_t + s (a t^{2 v - 1} s - b) \omega_s & = & - c s \omega + f (t) \label{omegat}.
\end{eqnarray}
We require solutions that satisfy the condition
\[ \lim_{t \rightarrow 0} \omega (t, s) = \pi (s) . \]

Sequentially, we have the following lemma:

\begin{lemma} \label{lemma1}
The solutions of the initial value problem
\begin{equation} \label{IVP}
\left\{ \begin{array}{lll}
\omega_t + s (a t^{2 v - 1} s - b) \omega_s  =  - c s \omega + f (t)\\
\lim_{t \rightarrow 0} \omega (t, s) = \pi (s)
\end{array} \right.
\end{equation}
where $a, b, c, v$ are constants with $a>0$ and $v > 0$, are given by
\begin{eqnarray}
& & \omega (t, s) = \label{IVPsol}\\ 
& & \left[\displaystyle{e^{- G \left( t ; \frac{1 - s a b^{- 2 v} \Gamma (2 v \nocomma, b t) e^{b t}}{s e^{bt}} \right) + G^{\star} \left( 0 ; \frac{1 - s a b^{- 2 v} \Gamma (2 v \nocomma, b t) e^{b t}}{s e^{bt}} \right)}}\right.\nonumber\\ 
& & \left.\cdot \pi \left( \frac{s e^{bt}}{1 - s e^{bt} a b^{- 2 v}  (\Gamma (2 v \nocomma, b t) - \Gamma (2 v \nocomma))} \right)\right] \nonumber\\
&  & \displaystyle{+ e^{- G \left( t ; \frac{1 - s a b^{- 2 v} \Gamma (2 v \nocomma, b t) e^{b t}}{s e^{bt}} \right)} \int_0^t \left\{ f (\tau) e^{G \left( \tau ; \frac{1 - s a b^{- 2 v} \Gamma (2 v \nocomma, b t) e^{b t}}{s e^{bt}} \right)} \right\}\, d \tau}\nonumber, 
\end{eqnarray}
where
\[ G (\mu_{\pm} ; \mathcal{C}) \assign \kern6pt\widetilde{\kern-7pt\int} \pm \frac{c}{a b^{- 2 v} \Gamma (2 v, b \mu) e^{b \mu} + e^{b\mu} \mathcal{C}}\, d \mu \]
and
\[ G^{\star} (\nu_{\pm} ; \mathcal{C}) \assign  \left(\kern6pt \widetilde{\kern-7pt\int} \pm \frac{c}{a b^{- 2 v} \Gamma (2 v, b \mu) e^{b \mu} +e^{b\mu} \mathcal{C}}\, d \mu \right)_{\mu = \nu}. \]
(Remark that the functions $G$ and $G^{\star}$ are merely notations for the integral of its respective function specified above.)
\end{lemma}

\begin{proof}
First, we note that $a, b, c, v$ are constants with $a>0$ and $v > 0$. We suppose, during the derivation, that $b \neq 0$ to avoid ambiguities. The final results will also hold for $b = 0$ since $v \gneqq 0$.
  
The characteristic equations of equation (\ref{IVP}) are
\[ d t = \frac{d s}{s (a t^{2 v - 1} s - b)} = \frac{d \omega}{f (t) - c s \omega} . \]
The first characteristic equation can be solved as follows:

First, rewrite the differential equation as
\[ \frac{d s}{d t} = s (a t^{2 v - 1} s - b) = a t^{2 v - 1} s^2 - b s. \]
Notice that this is a Bernoulli's Equation in $s$. So, we let $\displaystyle{z \assign \nicefrac{1}{s} \Longleftrightarrow s = \nicefrac{1}{z}}$. With this change of variable, we may rewrite the above equation as
\begin{eqnarray}
\frac{d z}{d t} - b z = - a t^{2 v - 1}. \label{reode}
\end{eqnarray}
Using variation of parameters technique, we can solve equation (\ref{reode}) as follows: note that the integrating factor here is $\displaystyle{{\mu} (t) = e^{\int{- b\, d t}} = e^{- b t}}$, and, so,
\begin{eqnarray}
& & \frac{d z}{d t} - b z = - a t^{2 v - 1} \nonumber \\
& \Longrightarrow & e^{- b t}  \left[ \frac{d z}{d t} \right] + z \left[ \frac{d}{d t} (e^{- b t}) \right] = e^{- b t}  (- a t^{2 v - 1}) \nonumber\\
& \Longrightarrow & \int d (e^{- b t} z) = \int e^{- b t}  (- a t^{2 v - 1})\, d t \nonumber\\
& \Longrightarrow & z = a b^{- 2 v} \Gamma (2 v \nocomma, b t) e^{b t} + e^{bt} C_1 \label{fscharacsol}, 
\end{eqnarray}
where $C_1$ is an arbitrary constant, and
\[ \Gamma (q_0, p_0) \assign \int_{p_0}^{\infty} x^{q_0 - 1} e^{- x}\, d x\]
is the \textit{``upper'' incomplete gamma function}. (Remark that when $p_0 = 0$, 
\[\Gamma (q_0, p_0) = \Gamma (q_0, 0) = \Gamma (q_0),\] 
where $\Gamma (q_0)$ is the \textit{usual gamma function}.) Substituting $\displaystyle{z = \nicefrac{1}{s}}$ back to the equation (\ref{fscharacsol}), yields the solution
\begin{eqnarray*}
\frac{1}{s} = a b^{- 2 v} \Gamma (2 v \nocomma, b t) e^{b t} + e^{bt} C_1 & \Longleftrightarrow & s = \frac{1}{a b^{- 2 v} \Gamma (2 v \nocomma, b t) e^{b t} + e^{bt} C_1}
\end{eqnarray*}
as required.
  
We further note that from the solution $s$ above, 
\[C_1 = \displaystyle\frac{1 - s a b^{- 2 v} \Gamma (2 v \nocomma, b t) e^{b t}}{s e^{bt}}\] 
{\emph{and}}, when $t=0$, $\displaystyle{C_1 = \nicefrac{(1 - s a b^{- 2 v} \Gamma (2 v \nocomma))}{s}}$.
  
Now, given that $s = \displaystyle\nicefrac{1}{(a b^{- 2 v} \Gamma (2 v \nocomma, b t) e^{b t} + e^{bt} C_1)}$, we can solve the second characteristic equation as follows:
   
First, we substitute $\displaystyle{s = \nicefrac{1}{(a b^{- 2 v} \Gamma (2 v \nocomma, b t) e^{b t} + e^{bt} C_1)}}$ into the second characteristic equation to get
\begin{equation}
\frac{d \omega}{d t} + \frac{c}{a b^{- 2 v} \Gamma (2 v \nocomma, b t) e^{b t} + e^{bt} C_1} \omega = f (t) \label{scharac}.
\end{equation}
Recall that we will have the solution $\omega_1$ as well as the solution $\omega_2$, and the solution to this ODE will be $\omega = \omega_1 + \omega_2$. Thus, we have that
\[ \begin{split}
\frac{d \omega_1}{d t} + \frac{c}{a b^{- 2 v} \Gamma (2 v \nocomma, b t) e^{b t} + e^{bt} C_1} \omega_1 = 0\\ 
\Longrightarrow \int \frac{d \omega_1}{\omega_1} = \int - \frac{c}{a b^{- 2 v} \Gamma (2 v \nocomma, b t) e^{b t} + e^{bt} C_1}\, d t. 
\end{split}\]
This implies that
\[ \omega_1 = C_2 \exp \left( \kern6pt\widetilde{\kern-7pt\int} - \frac{c}{a b^{- 2 v} \Gamma (2 v \nocomma, b t) e^{b t} + e^{bt} C_1}\, d t \right), \]
where $C_2$ is a constant. For convenience, we define
\[ G (\mu_{\pm} ; \mathcal{C}) = \kern6pt\widetilde{\kern-7pt\int} \pm \frac{c}{a b^{- 2 v} \Gamma (2 v \nocomma, b \mu) e^{b \mu} +e^{b\mu} \mathcal{C}}\, d \mu. \]
Next, we multiply
\[ \exp \left( \kern6pt \widetilde{\kern-7pt\int} \frac{c}{a b^{- 2 v} \Gamma (2 v \nocomma, b t) e^{b t} + e^{bt} C_1}\, d t  \right) = \exp \left( G (t_+ ; C_1) \right) \]
throughout the equation (\ref{scharac}) [where $\omega$ is now interchanged with $\omega_2$ for trivial reasons] to get
\begin{eqnarray*}
& & \exp (G (t_+ ; C_1)) \nocomma \frac{d \omega_2}{d t} + \exp (G (t_+; C_1)) \nocomma \frac{d\,G (t_+ ; C_1)}{d t} \omega_2 \\ 
& & = \exp (G (t_+ ; C_1)) \cdot f (t) \\ 
& & \\
& \Longrightarrow & \omega_2 = \frac{\int_0^t f (\tau) \exp (G (\tau_+ ; C_1))\, d \tau}{\exp (G (t_+ ; C_1)) \nocomma}.
\end{eqnarray*}
The solution of equation (\ref{scharac}) is, therefore,
\begin{eqnarray}
\omega & = & \omega_1 + \omega_2 \label{scharacsol} \\
& = & \exp (G (t_- ; C_1)) \left\{ C_2 + \int_0^t f (\tau) \exp (G (\tau_+ ; C_1))\, d\tau \right\} \nonumber. 
\end{eqnarray}
  
Sequentially, let $C_2 = A (C_1)$, where $A (y)$ is an {\emph{arbitrary}} function, and substitute the values 
\[\displaystyle{C_1 = \frac{1 - s a b^{- 2 v} \Gamma (2 v \nocomma)}{s}}\] 
(i.e., $\displaystyle{C_1 = \nicefrac{(1 - s a b^{- 2 v} \Gamma (2 v \nocomma))}{s}}$ when $t = 0$) \emph{and} 
\[C_2 = A (C_1)\]
into the solution $\omega$ above (i.e., equation (\ref{scharacsol})) with the imposed initial condition
\[ \lim_{t \rightarrow 0} \omega (t, s) = \pi (s) \]
(i.e., $t = 0$). Then, we have
\[ \omega (0, s) = \exp \left( G^{\star} \left( 0_- ; \frac{1 - s a b^{- 2 v} \Gamma (2 v \nocomma)}{s} \right) \right) \cdot A \left( \frac{1 - s a b^{- 2 v} \Gamma (2 v \nocomma)}{s} \right), \]
where
\[ G^{\star} (\nu_{\pm} ; \mathcal{C}) \assign \left( \kern6pt \widetilde{\kern-7pt\int} \pm \frac{c}{a b^{- 2 v} \Gamma (2 v \nocomma, b \mu) e^{b \mu} +e^{b\mu} \mathcal{C}}\, d \mu \right)_{\mu = \nu} . \]
Equating $\omega (0, s)$ to $\pi (s)$ \emph{and} letting 
\[\displaystyle{y \assign \frac{1 - s a b^{- 2 v} \Gamma (2 v \nocomma)}{s} \Longrightarrow s = \frac{1}{y + a b^{- 2 v} \Gamma (2 v \nocomma)}},\] 
we find
\[ A (y) = \pi \left( \frac{1}{y + a b^{- 2 v} \Gamma (2 v \nocomma)} \right) \cdot \exp \left( - G^{\star} \left( 0_- ; y \right) \right) . \]
That is,
\begin{eqnarray*}
C_2 & = & A (C_1)\\
& = & A \left( \frac{1 - s a b^{- 2 v} \Gamma (2 v \nocomma, b t) e^{b t}}{s e^{bt}} \right)\\
& = & \pi \left( \frac{s e^{bt}}{1 - s e^{b t} a b^{- 2 v}  (\Gamma (2 v \nocomma, b t) - \Gamma (2 v \nocomma))} \right)\\ 
& & \cdot \exp \left( - G^{\star} \left( 0_- ; \frac{1 - s a b^{- 2 v} \Gamma (2 v \nocomma, b t) e^{b t}}{s e^{bt}} \right) \right)
\end{eqnarray*}
when $\displaystyle{C_1 = \nicefrac{(1 - s a b^{- 2 v} \Gamma (2 v \nocomma, b t) e^{b t})}{s e^{bt}}}$. Substituting $C_1$ and $C_2$ into the solution $\omega$ above (i.e., equation (\ref{scharacsol})), we get
\begin{eqnarray*}
\omega & = & \exp \left( G \left( t_- ; \frac{1 - s a b^{- 2 v} \Gamma (2 v \nocomma, b t) e^{b t}}{s e^{bt}} \right) \right)\\
&  & \cdot \pi \left( \frac{s e^{bt}}{1 - s e^{bt} a b^{- 2 v}  (\Gamma (2 v \nocomma, b t) - \Gamma (2 v \nocomma))} \right)\\ 
& & \cdot \exp \left( - G^{\star} \left( 0_- ; \frac{1 - s a b^{- 2 v} \Gamma (2 v \nocomma, b t) e^{b t}}{s e^{bt}} \right) \right)\\
&  & + \left[\exp \left( G \left( t_- ; \frac{1 - s a b^{- 2 v} \Gamma (2 v \nocomma, b t) e^{b t}}{s e^{bt}} \right) \right)\right.\\ 
& & \left.\cdot \int_0^t f (\tau) \exp \left( G \left( \tau_+ ; \frac{1 - s a b^{- 2 v} \Gamma (2 v \nocomma, b t) e^{b t}}{s e^{bt}} \right) \right)\, d \tau\right]. \\
  \end{eqnarray*}
After some algebraic manipulations, we can get equation (\ref{IVPsol}) as desired.
\end{proof}

\begin{remark}
Given
\begin{eqnarray*}
\Gamma (q_0, p_0) & \assign & \int_{p_0}^{\infty} t^{q_0 - 1} e^{- t}\, d t\\
\Longrightarrow \Gamma (2 v \nocomma, b t) - \Gamma (2 v \nocomma)
& = & \int_{b t}^{\infty} x^{2 v - 1} e^{- x}\, d x - \int_0^{\infty} x^{2 v - 1} e^{- x}\, d x,
\end{eqnarray*}
but since $b \in \mathbbm{R}$, we cannot manipulate the two integrals above further since the factor $b t$ can either be positive {\emph{or}} negative.
\end{remark}

\section{General Form of the Solutions to the Laplace Transforms of the Fokker-Planck Equation of Fractional Brownian Motion}\label{proofs}

We dedicate this section to impose conditions on the factors of equation (\ref{IVPsol}) since we now know that the Laplace transform of any solution $u (t, x)$ of equation (\ref{mainFPK}) {\emph{must}} be of the form of equation (\ref{IVPsol}).

\begin{lemma} \label{lemma2}
If equation (\ref{IVPsol}) represent the Laplace transform of a solution $u (t, x)$ then $f (t)$ is given by
\begin{eqnarray*}
e^{G^{\star} \left(0 ; - a b^{- 2 v} \Gamma (2 v \nocomma, bt)\right)} \pi \left( - \frac{1}{a b^{- 2 v}  (\Gamma (2 v \nocomma, b t) - \Gamma (2 v \nocomma))} \right) & & \\ 
+ \int_0^t \{ f (\tau) e^{G (\tau ; - a b^{- 2 v} \Gamma (2 v, b t))} \}\, d \tau & = & 0.
\end{eqnarray*}
\end{lemma}

\begin{proof}
By definition of a solution, $u (t, x)$ is, for fixed $t$, integrable near $x = 0$. Therefore $\omega (t, s) \rightarrow 0$ as $s \rightarrow \infty$. Note that we have, with even stronger reason, that
\[ \left| e^{G \left( t ; \frac{1 - s \tmop{ab}^{- 2 v} \Gamma (2 v, b t) e^{b t}}{s e^{bt}} \right)} \right| \omega (t, s) \rightarrow 0 \]
as $s \rightarrow \infty$ for all $c \in \mathbbm{R}$ [see Proposition \ref{prop1} in Remark \ref{propremark}].
\end{proof}

\begin{remark}\label{propremark}
We need to check if we can move the limit inside the integral sign of equation (\ref{IVPsol}). Hence, we proved the following proposition by using \textit{dominated convergence theorem}.

\begin{proposition} \label{prop1}
\begin{eqnarray*}
& & \lim_{s \rightarrow \infty} \int_0^t \left\{ f (\tau) e^{G \left( \tau ; \frac{1 - s a b^{- 2 v} \Gamma (2 v \nocomma, b t) e^{b t}}{s e^{bt}} \right)} \right\}\, d \tau\\ 
& = & \int_0^t \lim_{s \rightarrow \infty} \left\{ f (\tau) e^{G \left( \tau ; \frac{1 - s a b^{- 2 v} \Gamma (2 v \nocomma, b t) e^{b t}}{s e^{bt}} \right)} \right\}\, d \tau\\
& = & \int_0^t \left\{ f (\tau) e^{G \left( \tau ; \lim_{s \rightarrow \infty} \left\{ \frac{1 - s a b^{- 2 v} \Gamma (2 v \nocomma, b t) e^{b t}}{s e^{bt}} \right\} \right)} \right\}\, d \tau\\
& = & \int_0^t \{ f (\tau) e^{G (\tau ; - a b^{- 2 v} \Gamma (2 v \nocomma, b t))} \}\, d \tau
\end{eqnarray*}
for all $t$ and constants $a, b, v$ with $v > 0$.
\end{proposition}

\begin{proof}
First, we try to prove that
\[ \lim_{s \rightarrow \infty} G \left( \tau ; \frac{1 - s a b^{- 2 v} \Gamma (2 v \nocomma, b t) e^{b t}}{s e^{bt}} \right) = G \left( \tau ; \lim_{s \rightarrow \infty}  \frac{1 - s a b^{- 2 v} \Gamma (2 v \nocomma, b t) e^{b t}}{s e^{bt}} \right) . \]
In other words, we try to prove that
\begin{eqnarray*}
& & \lim_{s \rightarrow \infty} \kern6pt\widetilde{\kern-7pt\int} \left\{ \frac{s e^{bt} c}{1 + s a b^{- 2 v}  (\Gamma (2 v, b \tau) e^{b \tau} - \Gamma (2 v \nocomma, b t) e^{b t})} \right\}\, d \tau\\
&=& \kern6pt\widetilde{\kern-7pt\int} \lim_{s \rightarrow \infty} \left\{ \frac{s e^{bt} c}{1 + s a b^{- 2 v}  (\Gamma (2 v, b \tau) e^{b \tau} - \Gamma (2 v \nocomma, b t) e^{b t})} \right\}\, d \tau .
\end{eqnarray*}
However, this case is quite obvious to apply \textit{dominated convergence theorem} on any arbitrary interval of integration $I$ since 
\[\displaystyle{\left| \frac{s e^{bt} c}{1 + s a b^{- 2 v}  (\Gamma (2 v, b \tau) e^{b \tau} - \Gamma (2 v \nocomma, b t) e^{b t})} \right| \leqslant k^{b \tau},\,\,\, k>e\approx2.71\ldots}\]
for almost every $\tau \in I$.
  
Next, we try to prove that
\begin{eqnarray*} 
& & \lim_{s \rightarrow \infty} \int_0^t \left\{ f (\tau) e^{G \left( \tau ; \frac{1 - s a b^{- 2 v} \Gamma (2 v \nocomma, b t) e^{b t}}{s e^{bt}} \right)} \right\}\, d \tau\\ 
& = & \int_0^t \lim_{s \rightarrow \infty} \left\{ f (\tau) e^{G \left( \tau ; \frac{1 - s a b^{- 2 v} \Gamma (2 v \nocomma, b t) e^{b t}}{s e^{bt}} \right)} \right\}\, d \tau .
\end{eqnarray*}
Suppose $f \in L^1 (0, t)$, that is,
\[ \int_0^t | f |\, d \tau < \infty . \]
Now, we let $h_s (\tau) \assign f (\tau) e^{G \left( \tau ; \frac{1 - s a b^{- 2 v} \Gamma (2 v \nocomma, b t) e^{b t}}{s e^{bt}} \right)}$ and let
\[ h (\tau) \assign \lim_{s \rightarrow \infty} h_s (\tau) = f (\tau) e^{\lim_{s \rightarrow \infty} G \left( \tau ; \frac{1 - s a b^{- 2 v} \Gamma (2 v \nocomma, b t) e^{b t}}{s e^{bt}} \right)} \]
for almost every $\tau \in [0, t]$.
  
It is easy to observe that there exists a function $m : \tau \mapsto [0, t]$ (i.e., $m (\tau) \in [0, t] \subseteq \mathbbm{R}$), and $c \in \mathbbm{R}$ with $c > e \approx 2.71 \ldots$ such that $\displaystyle{h_s (\tau) = O (c^{m (\tau)})}$ (i.e., ``Big-O of $\displaystyle{c^{m (\tau)}}$''). That is, $| h_s (\tau) |$ is \emph{dominated} by some function which doesn't depend on $s$ for almost every $\tau \in [0, t]$. Therefore, it follows that
\begin{eqnarray*} 
& & \lim_{s \rightarrow \infty} \int_0^t \left\{ f (\tau) e^{G \left( \tau ; \frac{1 - s a b^{- 2 v} \Gamma (2 v \nocomma, b t) e^{b t}}{s e^{bt}} \right)} \right\}\, d \tau\\ 
& = & \int_0^t \lim_{s \rightarrow \infty} \left\{ f (\tau) e^{G \left( \tau ; \frac{1 - s a b^{- 2 v} \Gamma (2 v \nocomma, b t) e^{b t}}{s e^{bt}} \right)} \right\}\, d \tau,
\end{eqnarray*}
and we're done.
\end{proof}
\end{remark}

\begin{lemma} \label{lemma3}
If $f (t)$ is continuous (this condition can be replaced by boundedness since it is a necessary condition) for $t \geqslant 0$, then equation (\ref{IVPsol}) is the Laplace transform of a solution $u (t, x)$ with initial values $P (x)$. The solution preserves positivity at least whenever $f (t) \geqslant 0$.
\end{lemma}

\begin{proof}
Let $\mathfrak{L}^{-1}$ be the inverse Laplace transform operator. We have that, for fixed $t$ and $\tau$,
\begin{eqnarray*} 
\mathfrak{L}^{- 1} & \left\{e^{- G \left( t ; \frac{1 - s a b^{- 2 v} \Gamma (2 v \nocomma, b t) e^{b t}}{s e^{bt}} \right) + G^{\star} \left( 0 ; \frac{1 - s a b^{- 2 v} \Gamma (2 v \nocomma, bt)e^{bt}}{s e^{bt}} \right)} \right. & \\
& \left. \cdot \pi \left(\frac{s e^{bt}}{1 - s e^{bt} a b^{- 2 v}  (\Gamma (2 v \nocomma, b t) - \Gamma (2 v \nocomma))} \right) \right\} > 0 &
\end{eqnarray*}
\emph{and}
\[ \mathfrak{L}^{- 1} \left\{ e^{- G \left( t ; \frac{1 - s a b^{- 2 v} \Gamma (2 v \nocomma, b t) e^{b t}}{s e^{bt}} \right)} \int_0^t \left\{ f (\tau) e^{G \left( \tau ; \frac{1 - s a b^{- 2 v} \Gamma (2 v \nocomma, b t) e^{b t}}{s e^{bt}} \right)} \right\}\, d \tau \right\} > 0 \]
due to the reason that the domain of the function $\pi$ is $D_{\pi} = (0, \infty)$. In other words, the first factor in equation (\ref{IVPsol}) as well as the factor under the integral represent Laplace transforms of positive functions which have all regularity properties imposed on our solutions. If $P (x)$ is non-decreasing then $\pi (s)$ is completely monotonic (i.e., $\displaystyle{\pi^n (s)  (- 1)^n \geqslant 0}$). The argument of $\pi \left( \hspace{1em} \right)$ in equation (\ref{IVPsol}) is absolutely monotonic (i.e., 
\[\displaystyle{\frac{\partial^n}{\partial s^n} \left\{ \frac{s e^{bt}}{1 - s e^{bt} a b^{- 2 v}  (\Gamma (2 v \nocomma, b t) - \Gamma (2 v \nocomma))} \right\} \geqslant 0}\] 
[see Remark \ref{proofsupp}]). Hence, the factor $\pi$ is completely monotonic \emph{or} the difference of two such functions, according as $P (x)$, is non-decreasing or only of bounded variation. The product of two Laplace transforms are the convolution of the original functions (by the convolution theorem). Moreover, this operation preserves the regularity properties imposed on our solutions.
\end{proof}

\begin{remark}\label{proofsupp}
For this lemma, since the domain of the function $\pi$ is $D_{\pi} = (0, \infty)$, it follows that
\begin{eqnarray*}
& & 0 < \frac{s e^{bt}}{1 - s e^{bt} a b^{- 2 v}  (\Gamma (2 v \nocomma, b t) - \Gamma (2 v \nocomma))} < \infty \\ 
& \Longrightarrow & e^{bt} a b^{- 2 v}  (\Gamma (2 v \nocomma, b t) - \Gamma (2 v \nocomma)) < \frac{1}{s} \rightarrow 0 \text{ as } s \rightarrow \infty\\
& \Longrightarrow & e^{bt} a b^{- 2 v}  (\Gamma (2 v \nocomma, b t) - \Gamma (2 v \nocomma)) < 0.
\end{eqnarray*}
Now, let $\displaystyle{\psi \assign e^{bt} a b^{- 2 v}  (\Gamma (2 v \nocomma, b t) - \Gamma (2 v \nocomma)) < 0}$ then \[\displaystyle{\frac{s e^{bt}}{1 - s e^{bt} a b^{- 2 v}  (\Gamma (2 v \nocomma, b t) - \Gamma (2 v \nocomma))} = \frac{s e^{bt}}{1 - s \psi}}.\] Observe that for $n \in \mathbbm{Z}_{\geqslant 1}$,
\begin{eqnarray*}
\frac{\partial^n}{\partial s^n} \left\{ \frac{s e^{bt}}{1 - s \psi} \right\} = \frac{(- 1)^{n + 1} n! \psi^{n - 1} e^{bt}}{[(- 1)^n  (\psi s - 1)]^{n + 1}} = \frac{n! \psi^{n - 1} e^{bt}}{(1 - \psi s)^{n + 1}}>0
\end{eqnarray*}
for both odd \emph{and} even $n$ since $\psi < 0$. Additionally, since
\[ \frac{s e^{bt}}{1 - s \psi} > 0 \implies \frac{\partial^n}{\partial s^n} \left\{ \frac{s e^{bt}}{1 - s \psi} \right\} > 0
\]
for all $n = 0, 1, 2, \ldots$ where the derivative of order $0$ of the function is the original function itself.
\end{remark}

\begin{lemma}
If $P (x)$ is non-decreasing, then equation (\ref{IVPsol}) with $f (t) \equiv 0$ defines a non-negative solution $u (t, x)$ of
\[ u_t = (a t^{2 v - 1} x u)_{x x} - ((b x + c) u)_x, \,\, 0 < x < \infty \]
with initial values $P (x)$ and
\begin{equation}
\int_0^{\infty} u (t, x)\,dx \equiv P (\infty) \label{equivP}.
\end{equation}
(Norm preserving or `reflecting barrier' solution.) For this solution, we have
\begin{eqnarray}
\lim_{x \rightarrow 0} u (t, x) & = & e^{- G (t ; - a b^{- 2 v} \Gamma (2 v \nocomma, b t)) + G^{\star} \left(0 ; - a b^{- 2 v} \Gamma (2 v \nocomma, bt)\right)}\label{slimx0} \\
& & \cdot \pi \left( - \frac{1}{a b^{- 2 v} \left( \Gamma (2 v \nocomma, b t) - \Gamma(2v)\right)} \right).\nonumber
\end{eqnarray}
\end{lemma}

\begin{proof}
The existence of a non-negative solution is guaranteed by Lemma \ref{lemma3}. Equation (\ref{equivP}) is equivalent to saying that $\omega (t, 0) = \pi (0)$ and this follows trivially from equation (\ref{IVPsol}) (also for $b = 0$). In order to prove equation (\ref{slimx0}), note that as $s \rightarrow \infty$, the argument of $\pi$ tends to the finite value
\[\frac{- 1}{a b^{- 2 v} (\Gamma (2 v \nocomma, b t)-\Gamma(2v))}\]
(since $\Gamma (2 v \nocomma, b t)$, $\Gamma(2v)$ converges). Given the function $P^{\star} (x)$ where its Laplace transform is given by expression $\pi$ in equation (\ref{IVPsol}) this means that a finite mass is concentrated at $x = 0$. Near $x = 0$, therefore, the solution $u (t, x)$ behaves like the function belonging to the factor of $\pi$, and the lemma follows.
\end{proof}

\begin{remark}
Intuitively,
\begin{eqnarray*}
\lim_{x \rightarrow 0} u (t, x) & = & \lim_{s \rightarrow \infty} \left\{ e^{- G \left( t ; \frac{1 - s a b^{- 2 v} \Gamma (2 v \nocomma, b t) e^{b t}}{s e^{bt}} \right) + G^{\star} \left( 0 ; \frac{1 - s a b^{- 2 v} \Gamma (2 v \nocomma, bt) e^{bt}}{s e^{bt}} \right)} \right.\\
& & \left. \cdot \pi \left( \frac{s e^{bt}}{1 - s e^{bt} a b^{- 2 v}  (\Gamma (2 v \nocomma, b t) - \Gamma (2 v \nocomma))} \right) \right\} \\
& & \\
& = & e^{- G (t ; - a b^{- 2 v} \Gamma (2 v \nocomma, b t)) + G^{\star} \left(0 ; - a b^{- 2 v} \Gamma (2 v \nocomma, bt)\right)}\\ 
& & \cdot \pi \left( - \frac{1}{a b^{- 2 v} \left( \Gamma (2 v \nocomma, b t) + \Gamma(2v)\right)} \right).
\end{eqnarray*}
\end{remark}

\section{Solution to the Fokker-Planck Equation of Fractional Brownian Motion}\label{solution}

In this \emph{short} section, we present to the readers the \textit{fundamental solution} to equation (\ref{IVPsol}) along with a way to obtain the solution $u (t, x)$.

The \textit{fundamental solution} of equation (\ref{IVPsol}) occurs when we consider $\pi (s) = \displaystyle{e^{- s \xi}}$, where $\xi > 0$ is a parameter. That is, the fundamental solution is
\begin{eqnarray}
& & \omega (t, s ; \xi) \label{fundsol} \\
& = & \displaystyle{e^{- G \left( t ; \frac{1 - s a b^{- 2 v} \Gamma (2 v \nocomma, b t) e^{b t}}{s e^{bt}} \right) + G^{\star} \left( 0 ; \frac{1 - s a b^{- 2 v} \Gamma (2 v \nocomma, bt)e^{bt}}{s e^{bt}} \right)}} \nonumber\\
& & \cdot \displaystyle{e^{- \left( \frac{s e^{bt} \xi}{1 - s e^{bt} a b^{- 2 v}  (\Gamma (2 v \nocomma, b t) - \Gamma (2 v \nocomma))} \right)}} \nonumber\\
&  & + \displaystyle{e^{- G \left( t ; \frac{1 - s a b^{- 2 v} \Gamma (2 v \nocomma, b t) e^{b t}}{s e^{bt}} \right)} \int_0^t \left\{ f (\tau) \displaystyle{e^{G \left( \tau ; \frac{1 - s a b^{- 2 v} \Gamma (2 v \nocomma, b t) e^{b t}}{s e^{bt}} \right)}} \right\}\, d \tau} \nonumber, 
\end{eqnarray}
where $a, b, c, v, \xi$ are constants with $a>0$, $v > 0$, and $\xi > 0$.

Moreover,
\begin{eqnarray}
& & u (t, x ; \xi) \label{mainsol}\\ 
& = & \mathfrak{L}^{- 1} \left\{ e^{- G \left( t ; \frac{1 - s a b^{- 2 v} \Gamma (2 v \nocomma, b t) e^{b t}}{s e^{bt}} \right) + G^{\star} \left( 0 ; \frac{1 - s a b^{- 2 v} \Gamma (2 v \nocomma, bt)e^{bt}}{s e^{bt}} \right)} \right. \nonumber\\
& & \cdot e^{- \left( \frac{s e^{bt} \xi}{1 - s e^{bt} a b^{- 2 v}  (\Gamma (2 v \nocomma, b t) - \Gamma (2 v \nocomma))} \right)} \nonumber\\
&  & + \left[e^{- G \left( t ; \frac{1 - s a b^{- 2 v} \Gamma (2 v \nocomma, b t) e^{b t}}{s e^{bt}} \right)}\right. \nonumber\\ 
& & \left. \left. \cdot \int_0^t \left\{ f (\tau) e^{G \left( \tau ; \frac{1 - s a b^{- 2 v} \Gamma (2 v \nocomma, b t) e^{b t}}{s e^{bt}} \right)} \right\}\, d \tau \right] \right\} \nonumber, 
\end{eqnarray}
where $a, b, c, v, \xi$ are constants with $a>0$, $v > 0$, and $\xi > 0$.

However, due to the non-analyticity of equation (\ref{fundsol}), the inverse Laplace transform cannot be deduced using analytic methods. Instead, numerical implementations should be more suitable to extract numerical solutions in this case.

\section{Application to Cox-Ingersoll-Ross Model}\label{application}

In this section, we consider an application of the derived non-analytic solution of the Fokker-Planck equation of fBm to the CIR model. Specifically, we will use the derived equation (\ref{mainsol}) to obtain the transition probability density function of the stock price $S$ at time $T$ given its price at time $t$, where $t<T$. This section follows closely to \citep{Hsu}.

First, before we proceed to applying the derived non-analytic solution of the Fokker-Planck equation of fBm to the CIR model, it is customary for the author to present to the readers why this application is interesting. The reason this application is interesting is due to the fact that one of the earliest well-known applications of \textit{Two Singular Diffusion Problems} \citep{Feller} is its application to the \emph{standard} CIR model (i.e., CIR model governed by Brownian motion). Also, since CIR model describes the evolution of stock price (or any interest rate derivatives in general), the CIR model have been proven very useful in the area of financial mathematics. Thus, it is only appropriate to make this application our first priority.

Next, we proceed to the application of our result. Note that since fBm is the generalized version of the usual Brownian motion, we consider the CIR model governed by fBm which assumes that the stock price $S$ follows the process
\begin{eqnarray}
dS_t = \mu (S,t)\,dt + \sigma(S,t)\,dB^{H} \label{gCEV}, 
\end{eqnarray}
where $\displaystyle{dB^H}$ is the fBm with Hurst parameter $\displaystyle{H \in (0,1)}$, $b(S,t) = aS+h$ are the dividends in continuous stream that each unit of stock pays out,
\[\mu = rS - b(S,t) = rS - (aS+h),\]
$r$ is the risk-free interest rate, $\sigma(S,t) = \sigma\sqrt{S}$, and $\sigma$ is a positive constant. Then, we have 
\[dS = [(r-a)S-h]\,dt + \sigma\sqrt{S}\,dB^H\]
and, so, the transition probability density function $P=P(S_T | S_t, T>t)$ will satisfy the Fokker-Planck equation
\begin{eqnarray}
H t^{2H - 1}  \frac{\partial^2}{\partial {S_t}^2}(\sigma^2 S_T P) - \frac{\partial}{\partial S_T}[((r-a)S - h)P] = \frac{\partial P}{\partial t}.
\end{eqnarray}

Thus, we let 
\[ \left\{ \begin{array}{lll}
     a & \assign & H\sigma^2 > 0\\
     v & \assign & H > 0\\
     x & \assign & S_T\\
     \xi & \assign & S_t\\
     b & \assign & r-H\sigma^2\\
     c & \assign & -h\\
     t & \assign & \Delta T=(T-t)
   \end{array} \right. \]
to obtain
\begin{eqnarray}
& & P(S_T | S_t, T>t) \nonumber\\
& = & \mathfrak{L}^{- 1} \left\{ e^{- V \left( \Delta T ; \frac{1 - s H\sigma^2 (r-H\sigma^2)^{- 2 H} \Gamma (2 H \nocomma, (r-H\sigma^2) \Delta T) e^{(r-H\sigma^2) \Delta T}}{s e^{(r-H\sigma^2)\Delta T}} \right)} \right. \nonumber \\ 
& & \cdot \displaystyle e^{V^{\star} \left( 0 ; \frac{1 - s H\sigma^2 (r-H\sigma^2)^{- 2 H} \Gamma (2 H \nocomma, (r-H\sigma^2)\Delta T)e^{(r-H\sigma^2)\Delta T}}{s e^{(r-H\sigma^2)\Delta T}} \right)} \nonumber \\
& & \cdot e^{-\left( \frac{s e^{(r-H\sigma^2)\Delta T} S_t}{1 - se^{(r-H\sigma^2)\Delta T} H\sigma^2 (r-H\sigma^2)^{-2H}  (\Gamma (2H \nocomma, (r-H\sigma^2) \Delta T)-\Gamma (2H \nocomma))} \right)} \nonumber\\
& & + \left[\displaystyle{\int_0^{\Delta T}{ \left\{ f (\tau) e^{V \left( \tau ; \frac{1 - s H\sigma^2 (r-H\sigma^2)^{- 2 H} \Gamma (2 H \nocomma, (r-H\sigma^2) \Delta T) e^{(r-H\sigma^2) \Delta T}}{s e^{(r-H\sigma^2)\Delta T}} \right)} \right\}\, d \tau}}\right.\nonumber\\
& & \left. \left. \cdot e^{-V \left( \Delta T ; \frac{1 - s H\sigma^2 (r-H\sigma^2)^{- 2 H} \Gamma (2 H \nocomma, (r-H\sigma^2) \Delta T) e^{(r-H\sigma^2) \Delta T}}{s e^{(r-H\sigma^2)\Delta T}} \right)}\right]\right\} \left(S_T\right) \nonumber,
\end{eqnarray}
where $s$ is the Laplace variable,
\[ V (\mu_{\pm} ; \mathcal{C}) \assign \kern6pt\widetilde{\kern-7pt\int} \mp \displaystyle{\tfrac{h}{H\sigma^2 (r-H\sigma^2)^{- 2 H} \Gamma (2 H, (r-H\sigma^2) \mu) e^{(r-H\sigma^2) \mu} +e^{(r-H\sigma^2)\mu} \mathcal{C}}\, d \mu},\]
\[V^{\star} (\nu_{\pm} ; \mathcal{C}) \assign \left( \kern6pt\widetilde{\kern-7pt\int} \mp \displaystyle{\tfrac{h}{H\sigma^2 (r-H\sigma^2)^{- 2 H} \Gamma (2 H, (r-H\sigma^2) \mu) e^{(r-H\sigma^2) \mu} +e^{(r-H\sigma^2)\mu} \mathcal{C}}\, d \mu} \right)_{\mu = \nu} \]
and $\displaystyle{f(\tau)}$ satisfies Lemma \ref{lemma2} with the appropriate substitutions suggested above \emph{and} $\displaystyle{\pi(s)}$ is in the form $\displaystyle{\pi(s) = e^{-s S_t}}$. Moreover, remark that $\displaystyle{\Delta T}$ is treated as a variable prior to taking the inverse Laplace transform. The numbers associated with $\displaystyle{T}$ and $\displaystyle{t}$ are substituted \emph{after} taking the inverse Laplace transform. This is the application of our result and, thereby, concludes our section.

\section{Conclusion and Further Research} \label{conclusion}

The purpose of this section is to conclude the paper and provide readers with the  some related topics or unanswered questions worth to be further investigated.

In summary, the author have found the non-analytic solution to the general
form of the Fokker-Planck equation of fBm along with proving that it is a
solution to the Fokker-Planck equation of fBm as intended. That is, the author have found the solution to
\begin{eqnarray}
  u_t = (a t^{2 v - 1} x u)_{x x} - ((b x + c) u)_x, \,\, 0 < x < \infty, 
\end{eqnarray}
where $u = u (t, x)$ and $a, b, c, v$ are constants with $a>0$ and $v > 0$. The solution is
\begin{eqnarray}
& & u (t, x ; \xi) \\
& = & \mathfrak{L}^{- 1} \left\{ e^{- G \left( t ; \frac{1 - s a b^{- 2 v} \Gamma (2 v \nocomma, b t) e^{b t}}{s e^{bt}} \right) + G^{\star} \left( 0 ; \frac{1 - s a b^{- 2 v} \Gamma (2 v \nocomma, bt)e^{bt}}{s e^{bt}} \right)} \right. \nonumber \\ 
& & \cdot e^{- \left( \frac{s e^{bt} \xi}{1 - s e^{bt} a b^{- 2 v}  (\Gamma (2 v \nocomma, b t) - \Gamma (2 v \nocomma))} \right)} \nonumber\\
&  & + \left[e^{- G \left( t ; \frac{1 - s a b^{- 2 v} \Gamma (2 v \nocomma, b t) e^{b t}}{s e^{bt}} \right)}\right. \nonumber\\
& & \left. \left. \cdot \int_0^t \left\{ f (\tau) e^{G \left( \tau ; \frac{1 - s a b^{- 2 v} \Gamma (2 v \nocomma, b t) e^{b t}}{s e^{bt}} \right)} \right\}\, d \tau\right] \right\}\nonumber, 
\end{eqnarray}
where $\xi > 0$ is a constant,
\[ G (\mu_{\pm} ; \mathcal{C}) \assign \kern6pt\widetilde{\kern-7pt\int} \pm \frac{c}{a b^{- 2 v} \Gamma (2 v, b \mu) e^{b \mu} +e^{b\mu} \mathcal{C}}\, d \mu\]
\emph{and}
\[G^{\star} (\nu_{\pm} ; \mathcal{C}) \assign \left( \kern6pt\widetilde{\kern-7pt\int} \pm \frac{c}{a b^{- 2 v} \Gamma (2 v, b \mu) e^{b \mu} +e^{b\mu} \mathcal{C}}\, d \mu \right)_{\mu = \nu}. \]
We have also applied the derived solution to the CIR model, where the Brownian motion is replaced by fBm.

There are some unanswered questions in this research. One of the main topics worth pursuing is, of course, actually solving the integral equations in the solution. By solving the integral equations, equation (\ref{fundsol}) can become analytic. Consequentially, we will be able to solve equation (\ref{mainsol}). On one hand, in order to solve the integral equations, readers may try to expand the gamma functions (both the \textit{``upper'' incomplete gamma function} and the \textit{usual gamma function}) into a power series and then solve the integral equations. On the other hand, readers may also try to use numerical methods to solve the integral equations and, consequentially, use numerical methods to solve equation (\ref{mainsol}) as well if analytical methods are not helpful.

\section*{Acknowledgment}

The author greatly acknowledges Dr. Chatchawan Panraksa for his continued help and support during the process of drafting and publishing of this research paper.




\nocite{*}
\bibliographystyle{elsarticle-num}
\bibliography{elsarticle-template-num}

\begin{thebibliography}{1}
\expandafter\ifx\csname url\endcsname\relax
  \def\url#1{\texttt{#1}}\fi
\expandafter\ifx\csname urlprefix\endcsname\relax\def\urlprefix{URL }\fi
\expandafter\ifx\csname href\endcsname\relax
  \def\href#1#2{#2} \def\path#1{#1}\fi

\bibitem{Feller}
W.~Feller, \href{http://dx.doi.org/10.2307/1969318}{Two singular diffusion
  problems}, Ann. of Math. (2) 54 (1951) 173--182.
\newblock \href {http://dx.doi.org/10.2307/1969318}
  {\path{doi:10.2307/1969318}}.
\newline\urlprefix\url{http://dx.doi.org/10.2307/1969318}

\bibitem{1}
A.~Swishchuk, A.~Ware, H.~Li, Option pricing with stochastic volatility using
  fuzzy sets theory, in: Northern Finance Association Conference Paper,
  Northern Finance Association, 2008, pp. 1--15.

\bibitem{Unal}
G.~{\"U}nal, Fokker-{P}lanck-{K}olmogorov equation for f{B}m: derivation and
  analytical solutions, in: Mathematical physics, World Sci. Publ., Hackensack,
  NJ, 2007, pp. 53--60.

\bibitem{Hsu}
Y.~Hsu, C.~Lee, T.~Lin, Constant elasticity of variance option pricing model:
  Integration and detailed derivation, in: Handbook of quantitative finance and
  risk management, Springer, New York, NY, 2010, pp. 471--480.

\bibitem{12}
J.~C. Cox, J.~E. Ingersoll, Jr., S.~A. Ross,
  \href{http://dx.doi.org/10.2307/1911242}{A theory of the term structure of
  interest rates}, Econometrica 53~(2) (1985) 385--407.
\newblock \href {http://dx.doi.org/10.2307/1911242}
  {\path{doi:10.2307/1911242}}.
\newline\urlprefix\url{http://dx.doi.org/10.2307/1911242}

\end{thebibliography}
\end{document}